\documentclass[12pt]{article}
\usepackage[dvips]{graphicx}
\usepackage{color}
\usepackage{theorem}
\usepackage{fullpage}
\usepackage{verbatim}
\usepackage{amsmath}
\usepackage{amssymb}
\usepackage[hang,small,bf,it,center]{caption}
\usepackage{theorem}
\usepackage{graphics,graphicx}

\usepackage[boxruled,noline]{algorithm2e}

\newtheorem{theorem}{Theorem}[section]
\newtheorem{lemma}[theorem]{Lemma}
\newtheorem{proposition}[theorem]{Proposition}
\newtheorem{corollary}[theorem]{Corollary}
\newtheorem{conjecture}[theorem]{Conjecture}

\newenvironment{proof}[1][Proof]{\begin{trivlist}
\item[\hskip \labelsep {\bfseries #1}]}{\end{trivlist}}
\newenvironment{definition}[1][Definition]{\begin{trivlist}
\item[\hskip \labelsep {\bfseries #1}]}{\end{trivlist}}

\newcommand{\qed}{\square}

\begin{document}
\bibliographystyle{siam}

\title{Combinatorial Polytope Enumeration}
\author{Sandeep Koranne \\ Mentor Graphics Corporation \\ West Linn, OR 97068\\ \\  Anand Kulkarni\footnote{This work was supported by a NSF graduate research fellowship.} \\ Department of Industrial Engineering and Operations Research \\ University of California, Berkeley, 94720 \\ \\ sandeep$\_$koranne@mentor.com, anandk@berkeley.edu}
\date{February 9, 2009\footnote{Footnotes and minor revisions added August 2009.}}
\maketitle

\begin{abstract}
We describe a provably complete algorithm for the generation of a tight, possibly exact, superset of all combinatorially distinct simple $n$-facet polytopes in $R^d$, along with their graphs, $f$-vectors, and face lattices.  The technique applies repeated cutting planes and planar sweeps to a $d-$simplex.  Our generator has implications for several outstanding problems in polytope theory, including conjectures about the number of distinct polytopes, the edge expansion of polytopal graphs, and the $d$-step conjecture.





\end{abstract}

\section{Introduction}

A subset $P$ of $R^d$ is called a \emph{convex polyhedron} if it is the intersection of a finite number of half-spaces, or equivalently the set of solutions to a finite system of linear inequalities $Ax\leq b$ ~\cite{grunbaum-book, ziegler-book}. It is called a \emph{convex polytope} if in addition to these properties it is also bounded.  When we consider a specific dimension $d$, we refer to a $d$-polytope.  A $d$-polytope with $n$ irredundant constraints is an $(n,d)$-polytope.  The face poset of a polytope is the set of all faces of $P$ ordered by set inclusion. Two polytopes are isomorphic if their face posets are isomorphic. 

Given $P$ as a convex $d$-polytope in $R^d$, a \emph{face} of a polytope is the intersection of any subset of its half-spaces.  A subset $F$ of $P$ is called a face of $P$ if it is either $\emptyset$, $P$, or the intersection of $P$ with a supporting hyperplane (a hyperplane $h$ of $R^d$ is supporting $P$ if one of the two closed halfspaces of $h$ contains $P$).  The faces of dimension 0 are called \emph{vertices}, dimension 1 are called \emph{edges}, dimension $d-2$ are called \emph{ridges}, and dimension $d-1$ are called \emph{facets}.  The set of faces partially ordered by inclusion is termed the \emph{face poset} or \emph{face lattice}.

The \emph{edge-vertex graph} $G_P$ of a polytope $P$ is the graph formed by the vertices and edges of P.  A graph $G$ is polytopal if there is a polytope whose edge-vertex graph is $G$.   A polytope is \emph{simple} if each vertex in $G_P$ has exactly $d$ neighbors, or equivalently, if each vertex in $P$ lies at the intersection of exactly $d$ hyperplanes.  Say that two simple polytopes are combinatorially equivalent iff their edge-vertex graphs are isomorphic.  As such, any simple polytope is defined fully by its edge-vertex graph up to combinatorial isomorphism, and its full face lattice may be recovered in polynomial time from the graph itself.


Despite their importance in optimization, relatively little is known about many combinatorial aspects of polytope theory.  In particular, good bounds on the number of distinct polytopes in $(n,d)$ are unknown, and at one time it was believed that no efficient methods existed for generating the class of simple $(n,d)$ polytopes\cite{grunbaum-book}.  Additionally, questions about the maximum diameter of $(n,d)$ polytopes, their edge expansion, and their combinatorial properties have persisted for many years.  We describe some of these questions below:

\begin{itemize}
\item{\textbf{Polytope enumeration:} How many combinatorial types of $(n,d)$ polytopes exist?} 
\item{\textbf{Hirsch's conjecture:} What is the maximum diameter of an $(n,d)$ polytope?}
\item{\textbf{Mihail-Vazirani conjecture:} Which classes of polytopal graphs are good edge expanders?}
\item{\textbf{Polytope generation:} How can we generate all combinatorial classes of $n-$facet polytopes from $n-1$-facet polytopes?} 
\end{itemize}

We present a method to inductively enumerate all facets of a $(n,d)$ polytope in the hope of answering these questions, and present sketches of how some of these questions may be addressed using our approach.  





\section{Inductive Polytope Generation}
In this section, we develop an algorithm for generating the graphs of all combinatorially distinct polytopes of a given dimension, beginning with the $d$-simplex.  While it was understood since the 1960s that brute-force enumeration algorithms for polytopes must exist, previous methods had operated over too general a class to be useful in reasoning about polytope structure.  For instance, it was shown by Tarski that the problem was decidable using a theorem of first-order logic, and later efforts enumerated large supersets that contained polytopes as a special case, such as shellable complexes, or $d$-regular graphs \cite{grunbaum-book}, then discriminating between polytopal and nonpolytopal structures using matroids.  The primary approach to enumerating polytopes from scratch was first proposed by Euler, termed the ``beyond-beneath" technique.  This method was made rigorous by Grunbaum and Sreedharan (\cite{grunbaum1967enumeration}), and later exploited by Amos Altshuler and Ido Shemer in a series of papers (\cite{altshuler1974combinatorial,altshuler1984construction,altshuler1984enumeration}).  Unfortunately, because they generated a large superset of the class of polytopes, these algorithms had limited value in reasoning about combinatorial properties, as the properties in question do not hold on these broader entities.

The core idea behind our approach is as follows.  We prove that any $n$-facet polytope is the intersection of an $(n-1)$-facet polytope with an additional halfspace.  Consequently, we show that we can generate any $n$-facet polytopes by taking the appropriate intersection between a halfspace and a polytope with one fewer facet (``cutting'' the polytope). Similarly, we show that we can generate all $n$-facet polytopes by cutting $(n-1)$-facet polytopes with halfspaces in each possible way.  Because polytopes are considered combinatorially equivalent when their graphs are isomorphic, we show that we need not consider the infinite number of possible cuts for each polytope, but only the relatively small number of cuts which produce distinct graphs in the resulting polytope.  Finally, we determine how the graphs of polytopes are affected by cuts, allowing us to construct the graph of the $n$-facet polytope from the graph of an $(n-1)$-facet polytope.

\begin{centering}
\begin{figure}[h]
\center
\includegraphics[scale=.75]{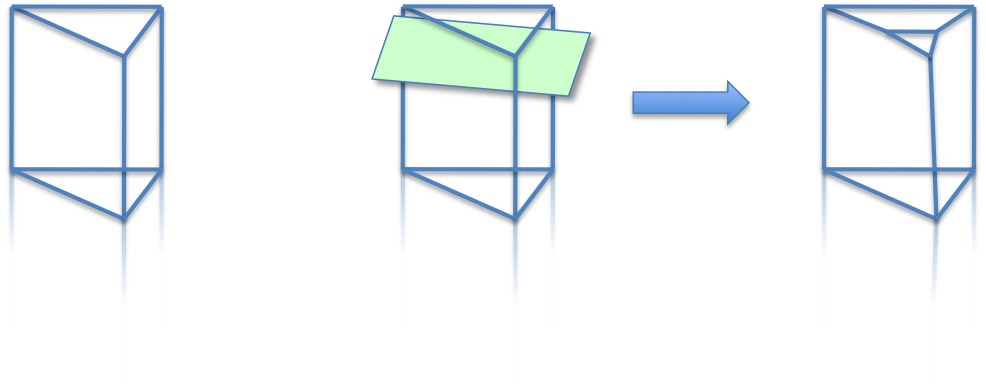}
\caption{Generation of an $n$-facet polytope from an $(n-1)$-facet polytope.}
\end{figure}
\end{centering}

\subsection{Proof of Completeness}
Define a \emph{polytope cut operation} as the intersection of a $d-$dimensional halfspace with an $n$-facet, $d$-dimensional polytope in $R^d$.  Let $H^{+}, \overline{H}$ and $H^-$ refer to a $d-$halfspace, its bounding hyperplane, and its complementary halfspace respectively.  Henceforth, $H$ will be used to mean either $H^{+}$ or $\overline{H}$, depending on if we refer to a halfspace or the hyperplane that bounds it.

Intersecting a halfspace $H$ with a polytope $P$ produces a new polytope $P'$. Although it is always possible to choose $H$ so that $P$ and $P'$ are combinatorially equivalent, because we are interested in polytopes without redundant constraints we discuss the cases where the two are combinatorially distinct.   Those vertices in $P'$ not present in $P$ are called \emph{new} vertices, and are generated by the intersection of $P$ and $\overline{H}$.  Those vertices shared by $P$ and $P'$ are called \emph{old} vertices.  Any vertices in $P$ not in $P'$ are \emph{removed} vertices.  The graph of removed vertices is called the cutset.  The set of new vertices constitutes a facet of $P'$ called the \emph{new} facet.

A facet $F$ of a polytope $P$, defined by halfspace $H_F$, is called \emph{projectively removable} if there exists a polytope $P'$ such that 1) $P$ and $P'$ are combinatorially isomorphic [ie, there exists a projective transformation $T$ such that $T(P)=P'$], and 2) $P'\ \backslash\ T(H_F)$ is a bounded set. 

We state the next lemma without proof, as it is drawn directly from a paper of Klee and Kleinschmidt \cite{klee}.
\begin{lemma}[Klee-Kleinschmidt]
Every simple polytope with $n>d+1$ facets has a projectively removable facet.
\end{lemma}

The next theorem lies at the core of our algorithm.

\begin{theorem}\label{kleelemma}
Every combinatorial type of $(n,d)$-polytope can be generated through the intersection of a $d$-halfspace $H^+$ with some combinatorial $(n-1,d)$-polytope, called its $parent$ polytope, except the simplex.
\end{theorem}
\begin{proof}

We simply apply the preceding lemma. Since every non-simplex polytope $P$ has a projectively removable facet $F$, then by the lemma there exists a combinatorially identical polytope $P'$ with facet $F'$ such that $(P' \backslash H_{F'})$ is a bounded set.  But if $(P' \backslash H_{F'})$ is bounded, then it is a polytope, and since exactly one hyperplane was removed this polytope must have exactly $n-1$ facets.  Thus $(P' \backslash H_{F'})$ is an $(n-1,d)$ polytope. Thus, we have that $P' = (P'  \backslash H_{F'}) \bigcap H_{F'}$, so that $P'$ is the intersection of an $(n-1,d)$ polytope and a $d$-halfspace.  Since $P$ and $P'$ have the same combinatorial type, we are done. 
\begin{flushright}
$\qed$
\end{flushright}
\end{proof}

Observe that the preceding claim is not true when we consider particular realizations of polytopes, rather than combinatorial types of polytopes.  For example, the regular cross-polytopes, such as a regular hypercube, cannot be generated via intersection of a halfspace with any combinatorially distinct polytope because each adjacent pair of facets is exactly orthogonal.  Thus, it should be clear that this method generates all combinatorially distinct graphs of polytopes, rather than any specific realizations of polytopes.

We observe that the following algorithm will generate the complete class of polytopes for a given dimension, proceeding inductively over the number of facets.

\begin{enumerate}
\item Generate all combinatorially nonisomorphic removable vertex sets $C=\{C_1,...,C_k\}$ for a given polytope's face lattice.
\item For each vertex set $C_i$, construct the face lattice of each new polytope $P_i'$ formed by separating $C_i$ from $P$ with a hyperplane.
\item Recurse on each new polytopal face lattice $P_i'$.
\end{enumerate}

\begin{centering}
\begin{figure}
\center
\includegraphics[scale=.45]{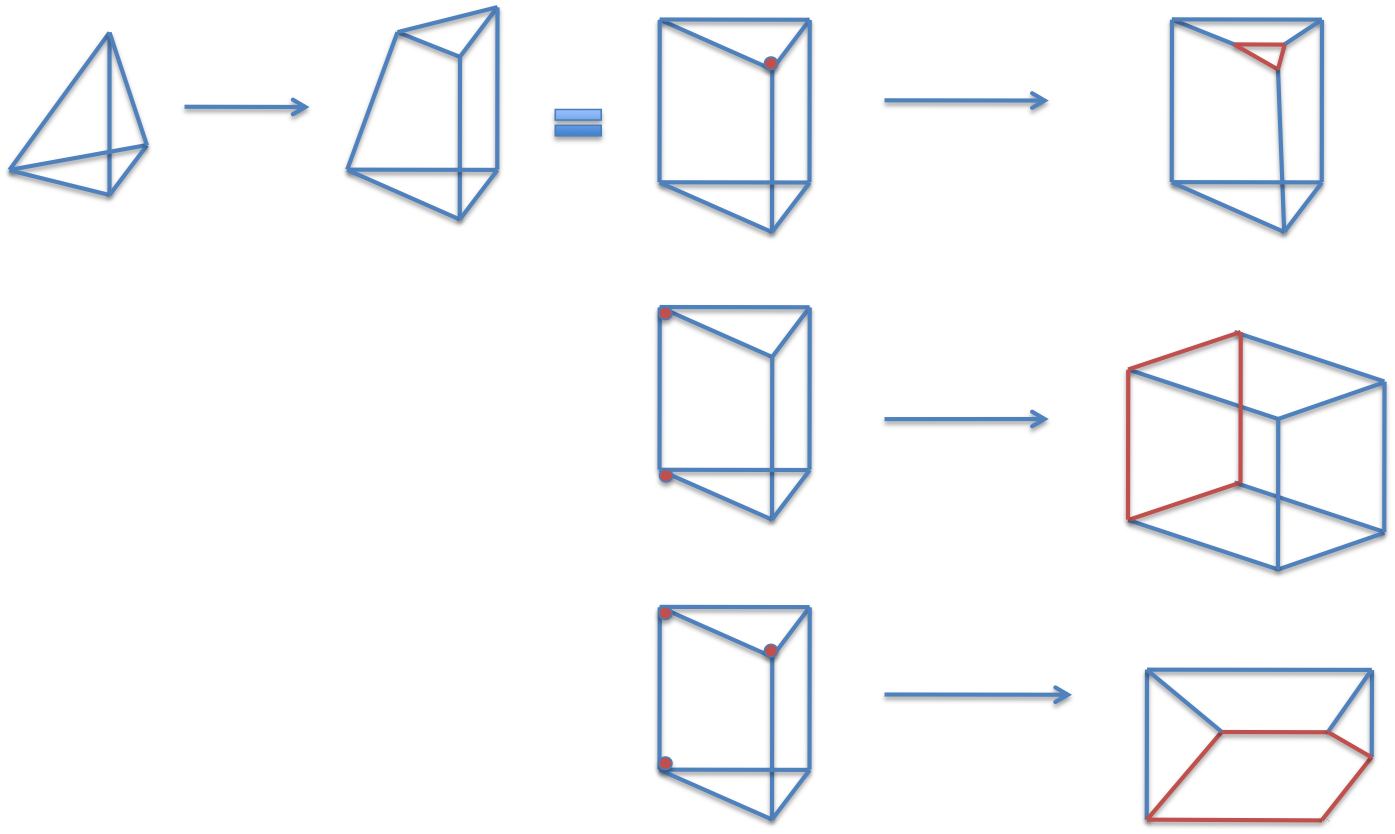}
\caption{Operation of the polytopal induction algorithm in three dimensions, up to $n=6$ facets.  Red dots indicate cutsets.  The top and bottom polytope are isomorphic.}
\end{figure}
\end{centering}

This infinite process may be implemented in the following algorithm, which terminates once all polytopes of a given facet count have been constructed. \\\\


\begin{algorithm}[H]
\caption{Algorithm to inductively generate all $n$-facet $d$-polytopes.}

\SetLine
Inputs: maximum facet count n, dimension d, with $n>d$\\
Output: list of all polytopes \\
Initialize P as the $d$-simplex\\
 \FuncSty{MakeAllChildPolytopes}(n,d,P) \Begin{
	\If { numFacets(P) = n } { return (P)}
	\Else {
	$C$ = GenerateAllRemovableCutsets($P$)\\
	\For{each removable cutset $C_i\in C$ }{
		let P' = CutPolytope($P$,$C_i$) \\
		polytopelist = polytopelist $\cup$ MakeAllChildPolytopes(n,d,P')
	}}
	{return polytopelist}
}
\label{enumerator}
\end{algorithm}\\\\

The runtime of this algorithm depends on the complexity of finding removable cutsets and the CutPolytope algorithm.  In the next two sections, we discuss each of these problems in detail.

\section{Generating cutsets}
Observe that at worst, we could consider every possible subset of the vertices of $P$ as a possible cutset.  However, this approach is both practically infeasible and theoretically unhelpful; it is evident that some subsets of vertices cannot be truncated by a common hyperplane without truncating additional vertices (for example, ones on opposite corners of a hypercube), so many of the resulting graphs would not correspond to any real polytope. This would provide limited utility in reasoning about graph-theoretic properties of polytopes.  Instead, in this section we establish necessary and possibly sufficient conditions for a subset $C$ of vertices of $G_P$ to form a cutset that can be realized by a hyperplane.  We then present an algorithm for generating these cutsets for a given polytopal graph.

Let $G_P$ be a polytopal graph.  A subset of vertices $C\subset G_P$ is a \emph{realizable cutset} if there exists a polytope $P$ with graph $G_P$ such that $C$ can be separated from $P$ by a hyperplane, generating a $n+1$-facet polytope $P'$.  Formally, $C\subset G_P$ is a realizable cutset if there exists a polytope $P'$ with graph $G_P$ and a halfspace $H$ such that for each vertex v of $P$ such that $v \in C$, $v\in H^-$, and for each vertex $v$ of $P$ such that $v \notin C$, $v \in H^+$.

We now provide three conditions that allow us to identify which vertices in a polytopal graph constitute a realizable cutset.

\begin{theorem}\label{props}
All realizable cutsets satisfy the following three properties.
\begin{enumerate}
 \item Connectedness: In $G_P$, every vertex in the cutset is connected to some other vertex in the cutset.
 \item Complementary connectedness: If removed, the cutset does not disconnect $G_P$.\footnote{Subsequent work on our part has shown the stronger condition that no subgraph corresponding to a face may be disconnected by the cutset.  This may be easily verified by a convexity argument.}
 \item Facet-free: The vertices of the cutset do not contain any facet of $P$.

\end{enumerate}
\end{theorem}
\begin{proof}
Suppose a cutset $C$ is realizable.  We will show that the three properties must hold:
\begin{enumerate}
\item Both $H^+$ and $P$ are convex sets. The intersection of two convex sets must be convex, and therefore, connected, so $H^+ \cap P$ is connected and convex.  Thus, the portion of the edge-vertex graph of $P$ lying in $H^+$ must be connected.  This is simply the cutset.

\item $H^-$ and $P$ are convex sets, so their intersection must be convex, and therefore connected.  Thus, the portion of the edge-vertex graph of $P$ falling in $H^-$ must be connected in $G_P$.  This is the complement of the cutset.  Thus the complement of the cutset must be connected, and therefore the cutset cannot disconnect the graph of P.

\item By definition, $P'$ has exactly one more facet than P.  If any facet of $P$ were contained in the cutset, the number of facets would decrease.  Hence no facets can be contained in the cutset.
\end{enumerate}
$\qed$
\end{proof}
Theorem \ref{props} establishes that by enumerating and cutting away only those cutsets which satisfy these conditions, we will generate a superset of all polytopal graphs.  However, it is also conceivable that some non-polytopal graphs will also be generated during this process.  The following conjecture suggests that this is not the case; ie, that these three conditions suffice to guarantee that the algorithm generates only polytopal graphs.

\begin{center}
\begin{figure}[h]
\center
\includegraphics[scale=.65]{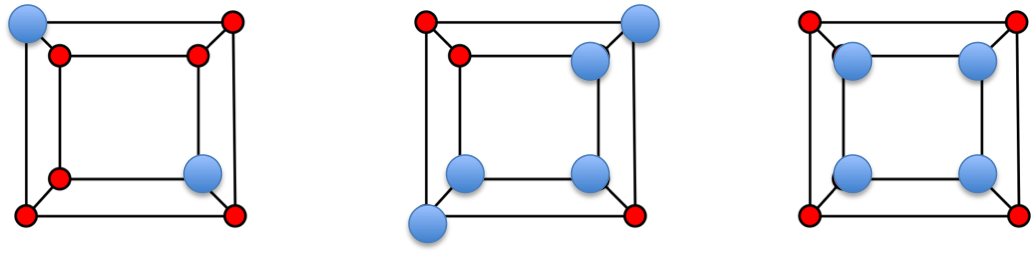}
\caption{Examples of violations of each of the three necessary conditions in theorem \ref{props}.  The large dots indicate sets of vertices that cannot be cut from any instance of the polytope by a hyperplane.}
\end{figure}
\end{center}

\begin{theorem}[Generalized Steinitz Theorem]
For $d\geq 4$, a simple face lattice is \emph{d-polytopal} if and only if it may be embedded as a polytopal subdivision of one of its facets in $R^{d-1}$.  For $d=3$, the lattice must also be $d$-connected.
\end{theorem}

The preceding statement is equivalent to Steinitz' theorem in three dimensions and follows from a recent result of Rybnikov about $d$-diagrams for $d\geq 4$ \cite{rybnikov}.  The theorem may be rephrased by saying that every simple $d$-connected polytopal subdivision of $R^{d-1}$ is $d$-polytopal.

\begin{lemma}
Any $k$-face of a polytope is removable.
\end{lemma}

\begin{theorem}
For $d=3$, every cutset satisfying properties 1 and 2 of Proposition \ref{props} is realizable.
\end{theorem}
\begin{proof}
In $d=3$, every polytope may be drawn as a $d$-connected planar graph by Steinitz' theorem.  Because the cutset $C$ is connected, it may be separated from the remainder of the graph $\overline{C}$ by a continuous closed non-self-intersecting planar loop $L$ passing through the edges $E=\{(x,y):x\in C, y\in \overline{C} \}$.  We construct a new graph as follows: draw a new vertex $v'$ at each intersection between $L$ and $E$, so that $L$ is divided into a set of line segments $l_1,\cdots,l_{|E|-1}$ and  each edge $e\in E$ is split into two edges incident to $v'$.  Now delete the vertices contained in $C$ along with all of their incident edges.  Since none of the new edges $l_i$ intersect any other edges except at a vertex and since our initial graph $G$ was planar, the resulting graph $G'$ is also planar. By Fary's theorem, the curved line segments $l_i$ may be drawn as straight lines in an isomorphic graph, yielding a polytopal sbdivision.

In order to show that $G'$ is 3-connected, remove any pair of vertices $(u,v)$ in $G'$ along with their incident edges.  If both $u$ and $v$ are newly formed vertices, then since the removal of $C$ does not disconnect $G'$, the removal of any subset of the new vertices cannot disconnect $G'$.  If neither $u$ nor $v$ are new vertices, then since $G$ is 3-connected, removing $u$ and $v$ cannot disconnect $G$ and therefore cannot disconnect $G'$.   Last, if $u$ is old and $v$ is new, we observe that since $G$ is 3-connected removing $u$ from either $G'$ or $G$ leaves a graph that is 2-connected.  Hence $G'$ is 3-connected and planar, and so by Steinitz' theorem it is realizable as a polytope in $R^3$. \begin{flushright} $\qed$ \end{flushright}



\end{proof}







\begin{conjecture}
For $d\geq 3$, every cutset satisfying properties 1 and 2 of Proposition \ref{props} is realizable.
\end{conjecture}

\begin{proof}

We will sketch two possible proofs, but omit the details.  Suppose the three conditions are satisfied, and construct the edge-vertex graph of $P$ and the edge-vertex graph of $P'$, according to the rules above.  It is sufficient to give a polytopal embedding of the graph $G_{P'}$ in $R^d$.  There are several possible approaches to solving this problem.  

First, we can assume that we have a realization of $P$ available, and identify an appropriate projective transformation that converts $P$ into $P'$ so that a cut is possible.  This can be achieved by constructing the appropriate linear separability problem of finding a hyperplane separating $C$ from $P$.  Rather than fully specifying the coordinates of each vertex in this problem, we begin with the fixed coordinates of $P$ and add open parameters representing a projective transformation of $P$.  The resulting problem is a convex program.  Last, we would need to show that when these three conditions are met, this convex program always has a feasible solution.

As an alternative, we can make use of the Generalized Steinitz Theorem, which gives exact conditions for recognizing when a face lattice (and by extension, a simple graph) is polytopal.  We can then prove the main conjecture by showing that when all of the three conditions are met, the appropriate embedding in $d-1$-space is preserved.

\begin{flushright}
$\qed$
\end{flushright}
\end{proof}


As a practical matter, generating a complete list of combinatorially nonisomorphic cutsets for a given polytope satisfying these three conditions is readily achieved by the following procedure, which can be suitably optimized to eliminate redundancies:

\begin{enumerate}
\item Initialize $V$ as the set of all rotationally nonisomorphic vertices of $P$, computed, for instance, via an automorphism check \cite{nauty}.  In particular, let each vertex be contained in a singleton set, so that $V=\{\{v_1\},...,\{v_k\}\}$
\item For each set $U\in V$: 
\begin{itemize}
\item Check if $U$ violates any of the three conditions. If so, skip to the next $U$.  If $V$ is exhausted, terminate.
\item Add $U$ to the list of cutsets.
\item Identify each nonisomorphic vertex neighbor $u$ of $U$.
\item For each such $u$, recursively run step 2 on the set $\{U \cup u\}$.
\end{itemize}
\end{enumerate}


\section{Cut algorithm}

In this section, we discuss how to construct the graph $G_{P'}$  and face lattice of a polytope resulting from separating a cutset $C$ from a polytope $P$ using a hyperplane $H$.  While we restrict our discussion here to simple polytopes, it is apparent how the algorithm can be extended to non-simple polytopes.

Observe that all vertices of $P$ not in the cutset are present in $P'$ and that all vertices in the cutset are not present in $P'$.  Additionally, the polytope $P'$ necessarily contains several new vertices generated by the intersection of $H$ with $P$. Since a vertex is a $0$-dimensional face of a polytope, it is generated by the intersection of exactly $d$ hyperplanes in a polytope.  Equivalently, it is generated at the intersection of a hyperplane with a $1$-dimensional face.  Consequently, a new vertex appears at every intersection of $H$ with an edge of $P$.  This occurs exactly along the edges with one vertex in $C$ and one vertex in $P\ \backslash C\ $.  Let this number be denoted (as in the definition of edge expansion) as $\delta(C)$.  Then, the intersection of $H$ with $P$ generates exactly $\delta(C)$ new vertices.  The new vertices form a new facet of $P'$; as such, we must add edges between the new vertices corresponding to faces of the new facet.  Finally, that the portions of $G_P$ that are not adjacent to members of $C$ are unaffected by the cut operation, since do not come in contact with the hyperplane $H$.

Thus, the procedure for constructing child graphs can be understood at a high level as follows.  Given a graph $G_P$ and a cutset $C$, to generate $G_{P'}$, perform the following steps on $G_P$:
\begin{enumerate}
\item{Create new vertices on each edge between $C$ and $G_P \backslash C$, forming a new facet}
\item{Delete vertices of the cutset and all incident edges}
\item{Add new edges between new vertices reflecting higher-dimensional faces of the new facet}.
\end{enumerate}

\begin{centering}
\begin{figure}[h]
\center
\includegraphics[scale=.60]{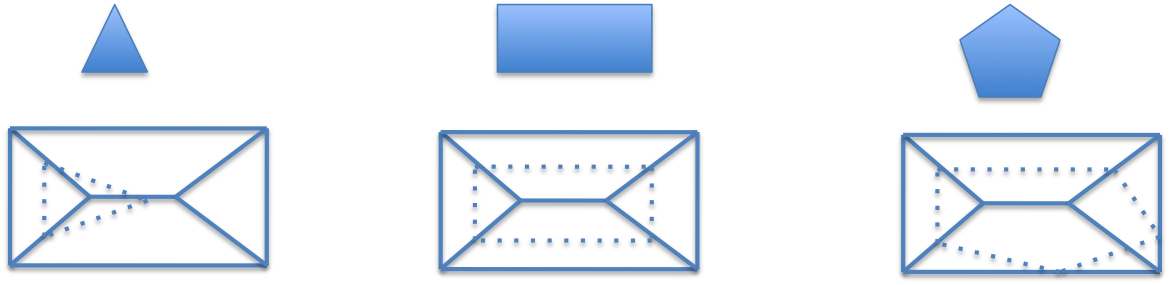}
\caption{New facets constructed by removing various cutsets from the $d=3$ simplicial prism.}
\end{figure}
\end{centering}

The details of step 3 require further development using the following lemma. 

\begin{lemma}\label{k2} (Face intersection lemma)
Define the parent face $f \in P$ of any face $f'\in P'$ as the face $f$ such that $f'=H\cap f$. New $k$-faces $u'$ and $v'$ lie on a common $k+1$-face in the new polytope if and only if their parent $(k+1)$-faces $u$ and $v$ lie on a common $k+2$-face in $P$.
\end{lemma}
\begin{proof}
Let $u$ and $v$ lie on a common $k+2$-face $F$ in $P$.  Then, $u$ and $v$ have $d-(k+2)$ facets in common.  Intersecting $H$ with $F$ increases the number of common facets of $u$ and $v$ by exactly 1, so $u'$ and $v'$ have $d-(k+1)$ facets in common.  Hence $u'$ and $v'$ lie on a common $k+1$-face.  Similarly, if $u'$ and $v'$ lie on a common $(k+1)$-facet after a cut, then they have $d-(k+1)$ facets in common, so removing $F$ causes their parent faces $u$ and $v$ to have $k$ facets in common.$\qed$
\end{proof}



\subsection{Multi-vertex cut algorithm}

Input: face lattice of polytope $P = \{V,F\}$, $C \subset V$\\
Output: face lattice of polytope $P' = \{V',F'\}$ generated by truncating $C$ from $P$\\
\begin{algorithm}[H]
\SetLine
 \FuncSty{CutPolytope}(P,C) 

	\For{each vertex $v$ in cutset $C$}
	{
		\For{each neighbor of $v$ not in $C$}
			{Add new vertex $v'$ to vertex list $V$}
		Delete old vertex $v$\\
			\For{each face $f$ containing $v$}
			{
			mark $f$ and recursively mark all parents of $f$ }
		}
	\For{each dimension $k$ from $d$ to $1$}
		{
			\For{each face $f$ of dimension $k$ in $F$}
					{
					\If{$f$ is marked}
						{
						create a new $(k-1)$-face $G$ in $F$\\
						\For{each pointer in $f$ to a face $g$ of dimension $k-2$}
							{
								\If{$g$ is marked} 
								{insert $g$ into $G$ }
							}
					}
			}
			return $P$
		}

\caption{Updating the face lattice to represent an additional facet (Cut Algorithm).}
\label{algo:main2}
\end{algorithm}


\subsubsection{Proof of correctness}
The algorithm operates as follows.  The input polytope is given as a set of faces.  The vertices (0-faces) are given as an explicit list of numbers; each $k$-face is given as a collection of pointers to $k-1$-faces.

First, add new vertices to the list of 0-faces, and mark them as changed.  For dimension 1, note any edges that contained a deleted 0-face, and mark them as changed. Then, for dimensions $k=1$ to $d$, mark which faces are changed. Next, counting backwards from dimensions $k=d$ to $k=1$, for each $k$-face $F$, look at all $(k-2)$-faces contained within it by following its pointers in the table.  For each $F$, add a new $(k-1)$ face $G$ consisting of all changed $(k-2)$-faces of $F$, and add $G$ as a new element of $F$.  Finally, remove all references to deleted faces.

The correctness then follows immediately from Lemma \ref{k2}.

\subsubsection{Runtime analysis}
Implementation of the face poset as a doubly-linked list of pointers means that each $k-$ face is visited no more than once per visit to each face of a higher dimension.  Since these steps form the largest loops of the algorithm, it dominates the runtime.

The runtime of the cut-update algorithm is thus $O(max(f^{*})^2)$, where $f^*$ is the $f$-vector of $P$.  To be precise, the number of updates that must be executed at each step is bounded above by $O(\sum_{k=0}^d  (f^*_k*f^*_{k+2})$, the total of the products of the number of faces of dimension $k$ and $k+2$ for all $k$.

\subsection{Incremental updates by hyperplane perturbation}

The algorithm of the previous section provided a method for determining the combinatorial structure of a polytope when a particular cutset of vertices are truncated by a hyperplane.  However, the enumeration algorithm we have presented requires the truncation of successively larger cutsets one at a time, many of which differ by only one vertex.  As such, much of the face lattice will be unchanged between polytopes with similar cutsets.  We are therefore motivated to seek an algorithm that represents the effects of a \emph{planar sweep} on the combinatorial structure of a polytope, truncating successively larger sets of vertices each time.  We formulate this problem as follows: given a hyperplane $H$ forming facet $F_H$ in face lattice $P$, along with a vertex $v$ adjacent to $F_H$, how do we compute the polytope $P'$ obtained when $H$ is translated to truncate $v$?

The following algorithm permits us to generate the new structure incrementally from the old structure, without the need to recompute the full face lattice of the polytope.  The mechanism is mainly identical to the algorithm of the preceding section, except that we are aware of the structure of all new vertices and faces formed in advance -- the new vertices formed in place of $v$ become a $(d-2$)-simplex. \\

\begin{algorithm}[H]
\SetLine
Input: facet $F_H$, face lattice $P$, vertex $v$ adjacent to some vertex $u$ of $F_H$\\
Output: updated face lattice $P'$ where $H$ has been perturbed to truncate $v$\\
\FuncSty{PushFacet(P,F,v)}
\Begin{
\If{$v$ not adjacent to some vertex $u$ of $F$}{return ``input error''}
\For{each neighbor of $v$ not in $F_H$}
			{Add new vertex $v'$ to vertex list $V$}
	Delete old vertex $v$\\
\For{each face $f$ containing $v$}
			{
		mark $f$ and recursively mark all parents of $f$ }

\For{each dimension $k$ from $d$ to $1$}
		{
			\For{each face $f$ of dimension $k$ in $F$}
					{
					\If{$f$ is marked}
						{
						create a new $k-1$-face $G$ in $F'$\\
						\For{each pointer in $f$ to a face $g$ of dimension $k-2$}
							{
								\For{each marked $g$} 
								{insert $g$ into $G$ }
							}
					}
			}
		}

}

\caption{Algorithm to update a face lattice after perturbation of a facet (sweep).}
\label{algo:sweep}
\end{algorithm}

\subsubsection{Proof of correctness}
The preceding algorithm is identical to the algorithm \ref{algo:main2}, except with respect to how new vertices are created.  Rather than creating new vertices on every edge of the truncated vertex $v$ we create new vertices only on those neighbors of $v$ not already in $F_N$.  The correctness of the rest of the algorithm follows from theorem \ref{k2}.

\subsubsection{Runtime Analysis}
The steps of the algorithm require adding all the new vertices generated by the perturbation, connect them to old vertices, and update any higher-order faces containing the vertex.  There are exactly $d-1$ new vertices that neighbor $v$   Because we need only consider the faces containing this vertex, the runtime may be considerably faster: the global runtime is still $O(\sum_{k=0}^d  (f^*_k*f^*_{k+2})$, but $f^*$ is now the f-vector of $P$ restricted only to those faces containing $v$.

\subsubsection{Graph-only variant}
If we are only interested in updating the polytopal graph rather than the full face lattice, we conjecture we may be able to operate considerably faster.  If omit the update of higher-dimensional faces and restrict our attention only to the graph, we need only update the vertices and edges of $P$ associated with the local neighborhood of the removed vertex $v$. There are $d$ new vertices formed and a maximum of $d$ old edges they may potentially be connected to, for an overall runtime of  $O(d^2)$.  Iteratively applying this algorithm over the entire vertex set leads to a faster mechanism for the cut algorithm of the previous section, requiring total time $O(|V|d^2)$ to truncate $V$ vertices. 



\section{Implementation and Experimentation}
In this section, we provide a full code-level description of the entire enumeration algorithm.  The algorithms follow below.

\begin{algorithm}[H]
\SetLine
void \FuncSty{Polytope::isValidCut}(Cuts C) \Begin{
	Graph complement = GraphComplement( graph, C)\;
	\If { complement.notConnected() } { return false\;}
	\For{f in FacetVertexVector() }{
			\If { f - C ==  $\phi$ } { return false\; }
	}
	return true\;
}
\caption{Determining whether a set of vertices satisfies Theorem \ref{props}.}
\label{algo:is_valid}
\end{algorithm}

\clearpage

\begin{algorithm}[H]
\SetLine
void \FuncSty{ComputeCuts}(Polytope P, FixedSet F, Set$<$Cuts$>$ result) \Begin{
	Orbit orbits = \FuncSty{ComputeOrbits}(P, F)\;
	\For{o in orbits }{
		int e = \FuncSty{min}(o)\;
		\If { $\neg$ P.isNeighbor( F, e)} { continue\;}
		\If { $\neg$ P.isValidCut( F $\cup$ e ) } { continue\;}
		F = F $\cup$ e\;
		\If {result.find( F ) $\ne$ result.end() } { result.insert( F )\;
		ComputeCuts( P, F, result )\;
		}
	}
}
\caption{Algorithm to generate all nonisomorphic cutsets for a polytope.}
\label{algo:compute_cuts}
\end{algorithm}
\clearpage

\begin{algorithm}[H]
\SetLine
void \FuncSty{CutPolytope}(Polytope P, Cut K) \Begin{
	\For{ v in K } { P.markVertexRemoved( v )\;}
	\For{ e in P.edges } { 
	\If{ e[0] $\in$ K $\wedge$ e[1] $\in$ K} { P.markEdgeRemoved( e )\; } 
	\If{ e[0] $\in$ K $\vee$ e[1] $\in$ K} { newVertices = P.cutEdge( e, K )\; } 
	}
	// Propagate change bits

	\For{ N=2; N$<$D;++N} {
		FaceRef PreviousCurrentFaceColumn( P.getFace( N-1 ) )\;
		FaceRef CurrentFace( P.getFace( N ) )\;
		\For{ int i=0; i $<$ C.size(); ++i } {
				FaceIndexRef T( C[i] )\;
				\For{ int k=0; i $<$ T.size() and T.changeCode==0; ++k } {
					\If { PC[ T[k] ].changeCode $>$ 0} {T.changeCode = 1\;}
				}
			}
	}
	// now discover new faces in d-1, again start off with N=2\\
	\For{ N=2; N$<$D;++N} {
		FaceRef PreviousCurrentFaceColumn( P.getFace( N-1 ) )\;
		FaceRef CurrentFaceColumn( P.getFace( N ) )\; 
		\For{ int i=0; i $<$ C.size(); ++i } {
				FaceIndexRef T( C[i] )\;
				FaceIndex cbt\;
				\For{ int k=0; i $<$ T.size(); ++k } {
					\If { PC[ T[k] ].changeCode $>$ 0 } { cbt.push\_back( C[i][k]\;}
				}
			}
		AddDerivativeFace( P, N, cbt )\;
		// reset change code for PC
	}

}
\caption{CutPolytope algorithm.}
\label{algo:cut_polytope}
\end{algorithm}
\clearpage

\begin{algorithm}[H]
\SetLine
int \FuncSty{CheckPolytope}(Polytope P, int depth) \Begin{
	\If { $\neg$ Is2D( P ) } { return (0)\;}
	\If {$\neg$ UniquePolytope(P)} { return (0)\;}
	Set$<$Cuts$>$ vertex\_cuts = \FuncSty{ComputeCuts}(P)\;
	\For{c in vertex\_cuts }{
		Polytope remainder = P\;
		\FuncSty{CutPolytope}( remainder, c )\;
		\If {$\neg$ UniquePolytope( remainder ) } { continue\;}
        \FuncSty{CheckPolytope}( remainder, depth+1)\;
	}
	return (0)\;

}
\caption{Polytope enumeration algorithm.}
\label{algo:main0}
\end{algorithm}
\clearpage

To further establish the correctness of this algorithm, we generated all 6-facet polytopes in 3 dimensions and all 8-facet polytopes in 4 dimensions, then checked their graphs by hand against those published in the literature \cite{catalog}.  A screenshot of one of the generated graphs is shown below as an example.

\begin{centering}
\center
\begin{figure}[h]
\includegraphics[scale=.7]{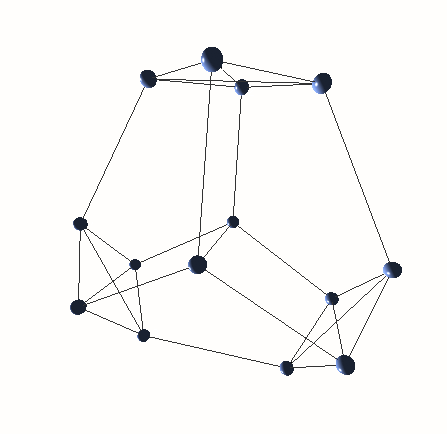}
\caption{The graph of a $d=4,n=8$ polytope generated by our algorithm.  This is a 4D simplex with three vertices truncated.}
\end{figure}
\end{centering}

\section{Applications}
The existence of an enumeration algorithm for polytopes that precisely specifies the behavior of polytopal graphs provides a new technique of addressing open questions about polytopes. Previous constructive operations like wedging and polytope products have permitted construction of specific polytopes satisfying certain properties, for instance, in the proof of the strong $d$-step conjecture for $d<5$ \cite{klee-walkup}.  However, because these operations were not guaranteed to generate all possible polytopes it was not possible to use them in arguments concerning global properties of polytopes.  

We have provided a mechanism to use the operations of \emph{cutting} and \emph{pushing} introduced in the previous two algorithms to generate all polytopes in a given dimension.  These operations enable inductive arguments over the entire class of simple polytopes.  In general, we will seek to show that some property of the polytopal graph is preserved during a cut or push operation.

\subsection{The $d$-step conjecture}
Recall that the Hirsch conjecture states that all $n-$facet polytopes in $R^d$ (henceforth (n,d)-polytopes) have diameter at most $n-d$.  According to a result of Klee and Walkup, the statement is equivalent to the conjecture that all simple $(d,2d)$-polytopes have diameter at most $d$, known as the $d$-step conjecture.

Klee and Walkup show that the maximum diameter of all $(d,2d)$ polytopes is realized by a particular class of simple $(d,2d)$-polytopes with two vertices $x$ and $y$ not sharing any common facet, known as Datnzig figures.  They observe further that the diameter of any Dantzig figure is realized by the distance between these two vertices.  Thus, to prove the $d-$step conjecture for a given $d$ it is sufficient to prove that the diameter between $x$ and $y$ on any $d$-dimensional Dantzig figure is at most $d$.

Our algorithm implies several sufficient statements for the $d$-step conjecture that we will present here, alongside a new proof that the $d-$step conjecture is true for all $d<6$. 

For the base case of an inductive proof we note that the $d-$ step conjecture is trivially true in dimension 2.  Next, we observe that any Dantzig figure $D$ is combinatorially equivalent to the intersection of a $d-$halfspace with a $(d,2d-1)$-polytope $P$, known as its \emph{parent polytope} (by theorem \ref{kleelemma}).  Recall that for any polytope with $n<2d$ every two vertices lie on a common facet \cite{grunbaum-book}.   Recall further that for any simple polytope with $n>d+1$ no facet can intersect every other facet, so no facet of $P$ has more than $2d-2$ $(d-2)$-faces \cite{ziegler-book}.  Recall last that the facets of a $d$-polytope are $(d-1)$-polytopes.  By the second and third recollections the facets of $P$ are at most $(d-1,2d-2)$-polytopes, and by the first fact the diameter of $P$ is at most the maximum diameter of any of its facets.  Thus the diameter of $P$ is at most $\Delta(d-1,2d-2)$.  By the inductive step, this means the diameter of $P$ is at most $(2d-2)-(d-1) = d-1$.

Thus, the d-step conjecture is true exactly if the diameter of every Dantzig figure $D$ is at most 1 above the diameter of its parent $P$.

$D$ contains a new facet $F$ that is not a facet of $P$.  Observe that, since $D$ is a Dantzig figure, the vertex $x$ must lie on $F$ and the vertex $y$ must not, or else $x$ and $y$ would lie on a common facet. 

\begin{proposition}\label{combin}
The combinatorial structure of $F$ is defined uniquely by the graph structure of the removed vertices.
In particular:
If $H$ truncates one vertex, $F$ is a $(d-1)$-simplex.
If $H$ truncates two vertices, $F$ is a $(d-1)$-simplicial prism.
If $H$ truncates three or more vertices, $F$ is a $(d-1)$ polytope whose type may be uniquely deduced from the face lattice of $P$ and the graph structure of the removed vertices.
\end{proposition}

\begin{proof}
Let $T$ denote the vertices of $P$ truncated by $H$ and let $S$ denote those vertices of $P$ not truncated by $H$.  Let $E(S,T)$ denote the number of edges between vertices in $S$ and vertices in $T$.

Because each vertex $v$ of $F$ is a 0-dimensional face it can occur only at the intersection of $H$ and an edge of $P$.  Since vertices may only be formed along the edges from $S$ to $T$, the number of vertices of $F$ is equal to the number of edges passing between truncated and untruncated vertices of $P$, or $E(S,T)$.

In particular, if $|T|=1$, then exactly $d$ vertices are incident to $T$, and so $F$ has exactly $d$ vertices.  Since the $(d-1)$-simplex is the only simple polytope with $d$ vertices, $F$ must be a $d-1$-simplex.  If $T$=2, then exactly $2d-2$ edges are incident to $T$, so $F$ has exactly $2d-2$ vertices.  A $d-1$-simplicial prism is the only simple $(d-1)$-polytope with $2d-2$ vertices, so $F$ must be a $(d-1)$-simplicial prism.  

In cases where multiple $(d-1)$-polytopes have the same number of vertices as $F$, it is necessary to explicitly compute the structure of $F$ based on the edges incident to $T$. As the combinatorial structure of a simple polytope is defined entirely by its edge-vertex graph, it suffices to account for all vertices and edges of $F$ to define its combinatorial structure.  In a simple polytope all vertices have exactly $d$ neighbors.  Because each vertex $v$ of $F$ is formed along an $S-T$ edge, exactly one neighbor of $v$ is a member of $T$.  Consequently, the remaining $(d-1)$ neighbors of $v$ must lie on $F$.  

The edges of $F$ may be determined by determining the intersection of $H$ with 2-dimensional faces of $P$, using either algorithm \ref{algo:sweep} or \ref{algo:cut_polytope}.
\end{proof}

\begin{corollary}
The number of combinatorially distinct facets generated by the truncation of $V$ vertices is equal to the number of possible connected $d$-regular graphs on $V$ nodes with distinct outgoing edges.
\end{corollary}

Call the vertices of $F$ \emph{new} vertices.  Call the remaining vertices of $D$ old vertices.

\begin{proposition}
The diameter of any Dantzig figure produced by the truncation of a single vertex $v$ is at most $d$. 
\end{proposition}

\begin{proof}
In order to have no common facets in $D$, $x$ must lie on $F$ and $y$ must not be a member of $F$. The distance from $v$ to $y$ in $P$ is at most equal to the diameter of $P$, or $d-1$. Since the vertices of $F$ are formed on edges incident to $v$, there is at least one vertex of $F$ a distance $d$ away from $y$.  Because it is a simplex, the diameter of a $F$ is 1. Thus, the distance from $y$ to $x$ is at most $d+1$.
\end{proof}

Consider now a \emph{translation sequence} of $k$ Dantzig figures $\{D_{T_i}\}_{i=1}^k$ generated by a linear translation of $H$ through $P$, where the $i$th figure results from the truncation of set $T_i$, with $|T_i|=i$ and $T_{i+1} = T_i \cup v$ for some vertex $v$ neighboring $T_{i}$ in the graph of $P$.

\begin{proposition}\label{translation}
Every Dantzig figure occurs as a member of some translation sequence.
\end{proposition}
\begin{proof}
Consider a Dantzig figure $D$.  By the Kleitman-Klee result, $D$ is the union of a halfspace $H$ and a parent polytope $P$.  $H$ must truncate some vertices of $P$.  Call this set $V_k$.  Since $H$ is a convex region and $P$ is a convex body, the vertices of $V_k$ must form a connected subgraph of $P$. 

By convexity, a linear translation of $H$ along its normal vector will bring it in contact with a vertex $v_k$ on the outermost boundary of the graph of $V_k$.  Thus, there is a translation of $H$ that will truncate a set of vertices $V_{k-1}$, where $V_{k}=v_k \cup V_{k-1}$.  Inductively, there are translations of $H$ that truncate sets $V_{k-2},V_{k-3},\ldots,V_{2},$ and $V_1$ where $V_{i}=v_{i} \cup V_{i-1}$ and $v_i$ is a neighbor of some vertex in $V_{i-1}$.  

The sequence $\{D_{V_i}\}_{i=1}^{k}$ is a translation sequence containing $D$. 
\end{proof}

We refer in the remainder of the section to the \emph{inductive step} of the $d$-step conjecture.  As such, we assume that the $d$-step conjecture is true for all dimensions up to $d-1$, and seek to show that it holds for dimension $d$ in all cases. 

\begin{theorem}
The inductive step of the $d$-step conjecture is true for all Dantzig figures whose largest facet has no more than $(2d-2)$ subfacets.
\end{theorem}
\begin{proof}
Immediate, by induction over $d$.  If the largest facet has no more than $(2d-2)$ subfacets then the every facet has diameter at most $(d-1)$.  Since, by the definition of a Dantzig figure, $x$ is one pivot away from a vertex that shares a common facet with $y$ the total distance from $x$ to $y$ is at most $d$.
\end{proof}

\begin{theorem}\label{induct}
The inductive step of the $d$-step conjecture is true\footnote{Subsequent work on our part has shown that many additional nontrivial cases arise that were not identified in this proof, particularly in higher dimensions.  The theorem may be shown to be valid even in these cases, but additional machinery is required.  These details are better addressed in a separate paper.} for all Dantzig figures such that $|T| \leq d$.
\end{theorem}

\begin{proof}
Consider now a Dantzig figure $D_1$ obtained by truncating a single vertex $v_1$ from a parent $(d,2d-1)$ polytope.  Observe that the distance from $y$ to $v_1$ is at most $d-1$; observe further that the vertices of $F$ are formed along the edges incident to $v_1$.  In particular, there exists at least one vertex $u_1$ on $F$ of distance $d-1$ from $y$.  Since $x$ lies on $F$ and $F$ is a $(d-1)$-simplex, $x$ lies at most distance 1 away from $y$.  

Thus a path of length $d$ exists from $x$ to $y$ passing from $x$ through $u_1$ to $y$.  Thus, the diameter of any such $D_1$ is at most $d$.

By the preceding theorem, we observe that any Dantzig figure $D_k$ may be obtained by translation of the hyperplane $H$ from an appropriately-chosen starting figure $D_1$.  Suppose that the diameter of some Dantzig figure $D_{k}$ in a translation sequence is $d$, and consider the Dantzig figure $D_{k+1}$ obtained from translating $H$ to truncate $v_{k}$, such that $k \leq d$. 

There are two cases to consider.

Case 1: If $v_{k}$ lies on the shortest path between $x$ and $y$ (say, along edges $(u,v_k)$ and $(v_k,w)$), then during the translation process, according to algorithm \ref{algo:sweep}, the following occurs:
\begin{enumerate}
\item{the vertex $v_{k}$ is replaced by a set of up to $(d-1)$ new vertices $\{v_k'\}$ with edges between them, since $H$ is perturbed to cross $(d-1)$ additional edges of $P$} 
\item{the edges $(v_k',w)$ are formed, since $v_k$ and $w$ lie on a common 2-face (indeed, a common edge)}
\item{the edges $(u_k,v_k')$ are formed for each $u_k$ connected to $u$, since the associated edges $(u,u_k)$ and $(v_k,w)$ share a $(d-2)$-face that includes the edge $u,v_k$}.
\item{the edges $(u,v_k)$ and $(v_k,w)$ are deleted since they refer to $v_k$}.
\end{enumerate}
Thus, the length of a shortest $x$-$y$ path passing through $v_k$ is unchanged, as it loses two edges and gains two edges.

Case 2: If $v_{k}$ does not lie on the shortest path between $x$ and $y$, but some vertex $u$ connected to $v_k$ does lie on the shortest $x$-$y$ path (say, using the path $t\rightarrow u \rightarrow s$, then according to algorithm \ref{algo:sweep} the following occurs:
\begin{enumerate}
\item{the vertex $v_{k}$ is replaced by a set of up to $(d-1)$ new vertices $\{v_k'\}$ with edges between them, since $H$ is perturbed to cross $(d-1)$ additional edges of $P$} 
\item{the edges $(v_k',w)$ are formed, since $v_k$ and $w$ lie on a common 2-face (indeed, a common edge)}
\item{the edges $(u_k,v_k')$ are formed for each $u_k$ connected to $u$, since the associated edges $(u,u_k)$ and $(v_k,w)$ share a $(d-2)$-face that includes the edge $u,v_k$}.
\item{the edges $(u,v_k)$ and $(v_k,w)$ are deleted since they refer to $v_k$}.
\end{enumerate}
Thus, the length of a shortest $x$-$y$ path passing through $u$ will increase by at most one, as the path through $(t,u,s)$ is replaced by the path $(t,u,v_k,s)$.  However, we can find an alternate path to vertex $s$ that bypasses $u$ entirely.  Observe that there exist $d-1$ neighbors of $x$ on the new facet $F_H$.  Observe that after the first perturbation, the new facet is a $(d-1)$-simplex, and so we have $d-1$ vertex-disjoint paths of equal length to reach vertex $s$ from $x$.  By induction, we may assume that not all of these paths are broken by a perturbation, since $T<d$, and we would need at least $d$ perturbations to affect a vertex from each path.  There thus exists at least one alternate path  $(x,x_2,...,u',s)$ whose length is not affected by this perturbation.

Thus, for at least the first $d$ vertices truncated, the $d-$step conjecture holds. $\qed$

\end{proof}

A recent preprint \cite{kettner} claims that the argument of the previous proof, which that author has derived separately from our work, is sufficient to imply the Hirsch conjecture in all cases.  The preceding argument was known to us for some time (eg, \cite{ieorseminar}), but is not sufficient to prove the Hirsch conjecture without additional assumptions. 

The ``push'' operation preserves the length of any path passing through the edge $e$ from $v_{k}$ to its neighbor $v' \in F$.  However, the ``push'' can strictly \emph{increases} the distance of any $x-y$ path passing through $v_{k}$ that does not pass through $e$.  This is illustrated in the figure below.  In particular, in \cite{kettner}, except in the cases where $(v',v_{k})$ lies on the diameter-defining path of a polytope, it is not implied that the diameter of a polytope is unchanged by a perturbation.

\begin{centering}
\begin{figure}[h]
\center
\includegraphics[scale=.60]{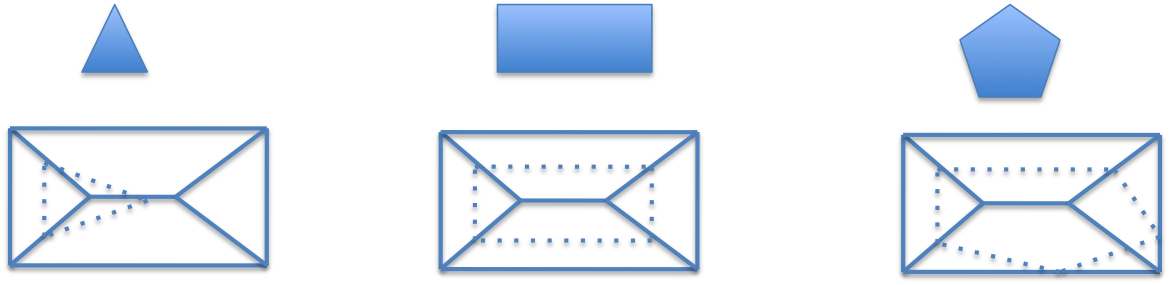}
\caption{A ``push" operation starting at a vertex increases the diameter of some paths going through that vertex and leaves others unchanged.  In particular, observe that the ``push" operation in the third picture increases the distance along a path.}
\end{figure}
\end{centering}

The preceding theorem preserved path length when the perturbed edge $v'-v_{k+1}$ lies on a shortest $x-y$ path.  In order for this argument to apply in cases $|T>d|$, it would need to be shown separately (for instance) that for every $v'$ neighboring $V_k$, is the member of \emph{some} shortest $x-y$ path in $D$.  This is not generally known to be true. 

\begin{theorem}
The $d$-step conjecture is true for $d<6$.
\end{theorem}

\begin{proof}
The number of vertices of a Dantzig figure that may be truncated by a hyperplane is strictly less than $|V|/2$, as any set of $|V|/2$ vertices disconnects a graph of size $|V|$, violating Proposition \ref{props}.  The upper bound theorem shows that $|V|/2 < d$ for any $(d,2d-1)$ polytope with $d<6$; hence, Theorem \ref{induct} implies the result.
\end{proof}


\begin{theorem}
This family of polytopes satisfy the $d$-step conjecture.
\end{theorem}

\subsubsection{Sufficient statements for the $d$-step conjecture}
In this section, we present three potential properties of polytopes that would be sufficient to imply the $d$-step conjecture.  

For each potential property, we discuss a conjecture, and then present a proof of the $d$-step conjecture stemming from that conjecture.


\begin{conjecture}
Every edge adjacent to the new facet $F_N$ in a translation sequence is the member of some shortest path from $x$ to $y$.
\end{conjecture}
\begin{proof}
We apply the argument of theorem \ref{induct}.  That theorem showed that all perturbations along a shortest path preserve their length.  Hence, if every possible perturbation is along a shortest path, the diameter of a $(d,2d)$ polytope is at most $d$.  
\end{proof}

\begin{definition}
A set of paths $\{X_1,X_2,...,X_N\}$ from $x$ to $y$ \emph{span} a (polytopal) graph if removing any one vertex $v_i$ from every path $x_i$ disconnects the (polytope's) graph. 
\end{definition}

\begin{conjecture} 
The set of $(x,y)$ paths of a $(d,2d-1)$ polytope of length $(d-1)$ or $(d)$ span the polytope for any pair $(x,y)$.
\end{conjecture}
\begin{proof}
If the set of minimal $(x,y)$ paths span the polytope, then any hyperplane removing all of them will disconnect the polytope, contradicting the requirement of proposition \ref{props} that no hyperplane cut can disconnect a polytope.  The preceding theorem implies that the set of all $(x,y)$ paths in a $(d,2d-1)$ polytope of length at most $d$ span the polytope.   Hence, no cut can cross all $d$-length paths.  Hence, at least one $(x,y)$ path of length at most $d$ is preserved in any Dantzig figure.
\end{proof}






\begin{corollary}
Conjectured: Spanning d-step conjecture: the set of minimal paths of $D$ span $D$.
\end{corollary}

\begin{conjecture}\label{contiguous}
No cutset $C$ may contain vertices $u$ and $v$ such that $u,v$ lie in a common facet of $F$ and there does not exist a connected path from $u$ to $v$ within $C\cap F$
\end{conjecture}

\begin{proof}
We prove the nonrevisiting path conjecture, that any two vertices of an n-facet polytope $P$ may be joined by a path that does not revisit any facet of $P$.

We induct over the number of facets of $P$.  Plainly, the theorem is true for a $d-$simplex.  Now suppose we have an $n-$facet polytope where the nonrevisiting path conjecture holds.  Consider the polytope $P'$ obtained from truncating a single vertex $v_1$.  By theorem \ref{combin}, $v_1$ is replaced by a new facet $F_N = \{v'_{1},...,v'_{d+1}\}$ that is a $d-$simplex in $P'$. 

Any nonrevisiting path containing the subpath $...v_0,v_1,v_2...$ passing through vertex $v_0$ may thus be replaced by a path $v_0,v'_{i},v'_{j},v_2$.  This path will remain nonrevisiting, as it contains one additional edge $v'_{i}\rightarrow v'_{j}$ and visits one additional new facet $F_N$ which was not visited.

Now consider a translation sequence as in Proposition \ref{translation} and suppose that polytope $P_k$ satisfies the nonrevisiting path conjecture.  Consider the vertex $v_k$ next cut during the translation and observe how $P_k$ is modified during the push as in algorithm \ref{algo:main2}.  New vertices $v_i'$ are formed along each edge $(v_i,v_k)$; each pair of vertices of $v_i'$ are connected a single edge and lie on the common facets $F_N$ and $F_k$.  We show that all subpaths through $v_k$ in $P_k$ remain nonrevisiting in $P_{k+1}$.  There are three subpaths to consider:

\begin{enumerate}
\item{$u_1,u_k,u_2$:  The subpath $u_1,u_k,u_2$ may be replaced by the subpath $u_1,v_1',v_2',u_2$.  This path has one additional edge $(v_1',v_2')$ and visits the additional facet $F'$ shared by vertices $v_1',v_2'$.  

It must also be shown that $F'$ has not already been visited by any path passing through $u_1,u_k,u2$ in $P$.  Suppose for a contradiction that it has.  Clearly, it was not visited by vertices $u_1$ or $u_2$ since they do not lie on $F'$. Therefore it must have been visited by some vertex $x$ in a distinct part of the path.  But if vertices $x$, $v_1'$,and $v_2'$ are all part of $F_N$ and $F'$, then they must lie on a common $(d-2)$-face of $P$, violating conjecture \ref{contiguous}.
}
\item{$u_1,u_k,v_k,v_1$: This path loses one edge and gains one edge, as discussed in the proof of theorem \ref{induct}.  Consequently, it replaces one newly-visited facet with another.}
\item{$v_1,v_k,v_2$: This path gains one additional vertex and one additional facet; however, by conjecture \ref{contiguous} this new facet has not yet been visited. }
\end{enumerate}

\end{proof}

\subsection{Closed-form bounds on the number of $(n,d)$ polytopes}

[pending]

\subsection{Edge expansion of polytopal graphs}
 Recall that the edge expansion of a graph $G=(V,E)$ is defined as 

\begin{equation}
h(G) = \min_{1\leq |S|\leq\frac{|V|}{2}} \frac{|\delta(S)|}{|S|}
\end{equation}

where $S\subset V$ and $\delta(S)\subset E$ is the set of edges with exactly one endpoint in $V$.  Recall also that the vertex expansion of $G$ is defined as

\begin{equation}
g(G) = \min_{1 \leq |A|\leq \frac{V}{2}} \frac{|\Gamma(A)|}{|A|}
\end{equation}

where $A \subset V$ and $\Gamma(A) \subset |V|$ is the set of vertex neighbors of $A$, ie, $\{v \in (V \backslash A):\exists u\in A, (u,v)\in E\}$.

It has been shown via the probabilistic method that there exist classes of polytopes with poor vertex expansion \cite{mihail}.  In particular, it is conjectured that removing a sub-linear number of vertices from any polytopal graph can substantially disconnect the graph \cite{kalai-book}.  Kalai gives a version of this conjecture as follows:

\begin{conjecture}
Let $P$ be a simple $d$-polytope with $m$ vertices.  There exists a subset $V'$ of vertices of $P$ such that:
\begin{enumerate}
\item $|V'| = O(n^{1-\frac{1}{d-1}})$
\item Removing $V'$ from $G_P$ separates $G_P$ into two parts, each with at least $n/3$ vertices.
\end{enumerate}
\end{conjecture}

Previous approaches to expansion problems have relied on transformations from well-understood families of polytopes.   Define a polytopal graph to be a good edge expander if its edge expansion is bounded below by a polynomial in $1/d$, where $d$ is the polytope's dimension.  Several families of polytopes have been shown to be good edge expanders by extending an argument that the hypercube is a good edge expander.  In particular, matching polytopes, independent set polytopes, and balanced matroid polytopes all have edge expansion at least 1.  The following theorem by Mihail summarizes the central argument that the hypercube is a good edge expander.

\begin{theorem}
The $d$-dimensional hypercube has edge expansion at least 1.
\end{theorem}
\begin{proof}
Consider the $d$-dimensional hypercube $P$ with vertex set $Q$.  Consider any subset $A$ of the vertices of $Q$ with $|A|\leq 2^{d-1}$.  We need to show that $|\Gamma(A)|\geq |A|$.  Consider all pairs of vertices $x_i,y_i \in Q$, and for each pair consider a shortest path $\rho_i$ chosen uniformly at random from all  shortest $x_i,y_i$ paths.  Observe that there are $2^{2d}$ paths $\rho_i$ in total, and that each path has average length $d/2$.  Observe further that the total number of edges in all the paths, when summed, is $d\cdot 2^d$.  Thus, each edge of the cube is a member of at most $2^{d-1}$ paths in each direction.
Next, note that the number of paths from $|A|$ to $\overline{A}$ that pass through $\Gamma(A)$ is $|A|\cdot|\overline{A}|$.  Consequently, 
\begin{equation}
|\Gamma(A)|\geq\frac{|A||\overline{A}|}{2^{d-1}} 
\end{equation}

Last, since $|A|\leq 2^{d-1}$, $\overline{A}\geq 2^{d-1}$, giving

\begin{equation}
|\Gamma(A)|\geq\frac{|A||\overline{A}|}{2^{d-1}} \geq |A| 
\end{equation}

which completes the proof.

\end{proof}

An open question posed by Mihail and Vazirani considers the edge expansion of $0/1$ polytopes \cite{feder}.  These are polytopes whose vertices lie only in the set $\{0,1\}^d$.  There are two related conjectures in this category, suggesting respectively that $0/1$ polytopes are either good or very good expanders.

\begin{conjecture}
There exists a polynomial $p(d)$ such that every 0/1 $d$-polytope has edge expansion at least $1/p(d)$. 
\end{conjecture}

\begin{conjecture}
Each 0/1 $d$-polytope has edge expansion at least 1.
\end{conjecture}

These latter conjectures are of interest because their existence would imply efficient randomized MCMC approximation algorithms for a number of open combinatorial and counting problems, as established by Alon, Sinclair, and Jerrum \cite{mihail,kaibel}. This technique was successfully applied by Jerrum and Sinclair to provide an efficient randomized algorithm for computing the permanent (using a special class of polytopal graphs known to expand well) and by Dyer, Frieze, and Kannan to measure the volume of a convex body\cite{sinclair,dyer}.  The associated argument is easy to understand.  It was established by Alon, Sinclair, and Jerrum that random walks on $d$-regular expander graphs possess the \emph{rapid mixing} propety, that is, they converge arbitrarily close to the stationary distribution of the walk after at most a polynomial (in $d$) number of steps \cite{noga, sinclair}.  Consequently, if polytope graphs are good expanders, then there exists an efficient algorithm to sample the vertices of a polytope (under certain mild assumptions).  Finally, methods have been developed to use efficient sampling algorithms for a combinatorial population to efficiently measure its size via MCMC.  Because several counting problems can be modeled as particular 0/1 polytopes, efficient sampling of these polytopes can produce efficient randomized algorithms for these problems.  Many counting problems, such as estimating network reliability, can be efficiently solved by positive resolution of the Mihail-Vazirani conjecture and applying the preceding method to the class of matroid polytopes. 

This technique was successfully applied by Jerrum and Sinclair to provide an efficient randomized algorithm for computing the permanent (using a special class of polytopal graphs known to expand well) and by Dyer, Frieze, and Kannan to measure the volume of a convex body\cite{sinclair,dyer}.  Many counting problems, such as estimating network reliability, can be efficiently solved by positive resolution of the Mihail-Vazirani conjecture and applying the preceding method to the class of matroid polytopes.  While 0/1 matroids have not yet been shown to expand, Kaibel showed that several additional subclasses of 0/1 polytopes are good expanders, including hypersimplices, stable set polytopes, and all simple 0/1 polytopes \cite{kaibel}.  

We also propose here that bounding the edge expansion may have positive implications for randomized algorithms for linear programming, such as those presented recently by Kelner and Spielman \cite{recentpaper}.  In particular, it may be possible to identify new families of polytopes where randomized algorithms converge in strongly polynomial time. %

Most successful approaches to these conjectures have been non-constructive; either they rely on the probabilistic method to prove the existence of polytopal graphs with particular properties, or they operate by reducing certain families of polytopes to known polytopes.  Part of the reason is that polytopal graphs are difficult objects to consider in their generality; while several classes of polytopes are well-understood, particularly in low dimensions, mechanisms for reasoning about higher-dimensional polytopes as a class do not exist. In the preceding sections of the paper, we have developed a new approach to edge expansion conjectures on polytopes that does not rely on transformation or probabilistic reasoning, but which explicitly constructs the class of all polytopal graphs by induction.  We believe that this reasoning may be usefully applied to address the the Mihail-Vazirani edge expansion conjecture, and to a lesser extent, Kalai's vertex expansion conjecture.

First, we note the obvious experimental method enabled by the algorithm; in the event that we seek a counterexample of a given dimension to some conjecture, we may simply run the polytopal graph generator and eventually we will be guaranteed to generate the corresponding counterexample, if it exists.  However, as the number of distinct polytopal graphs grows very rapidly with $n$ and $d$, this is not a practical way to produce counterexamples at very high dimensions.  Nonetheless, there are many nontrivial questions that may be resolved by using a brute-force search across low dimensions, such as diameter conjectures and verifying Hamiltonicity. Our survey suggests there exists something of a cottage industry in the literature of deriving single-instance results of the form ``There exists at least one five-dimensional polytope with fewer than 500 vertices with property X''.  This may be in part due to the lack of knowledge about the structure of higher-dimensional polytopes.  Indeed, a relevant example is given by Kaibel \cite{kaibel}, who proves that the Mihail-Vazirani conjecture holds on all 0/1 polytopes of dimension $d \leq 5$.  We observe that this result might have been easily produced computationally using our own method.

Theoretical approaches enabled by the induction algorithm are of greater interest with regard to verifying the two conjectures given  above.  The basic approach enabled by this algorithm will be the same as that proposed for the $d$-step conjecture; we simply carry out induction over the set of all polytopes according to the number of facets, and bound the change in the expansion of the graph observed during the inductive step.  

\section{Future Work}
Induction over polytopes appears to be a powerful method.  Upcoming versions of this paper may identify the appropriate enumerative complexity class of the polytope enumeration problem and discuss variants that generate nonsimple polytopes and 0/1 polytopes.




\bibliography{polytope}

\end{document}